\documentclass[10pt]{article}
\usepackage{amsmath,amsthm,amsfonts,latexsym,amsopn,verbatim,amscd,amssymb}
\theoremstyle{plain}
\newtheorem{theorem}{Theorem}[section]
\newtheorem{lemma}[theorem]{Lemma}
\newtheorem{corollary}[theorem]{Corollary}
\newtheorem{definition}[theorem]{Definition}
\newtheorem{proposition}[theorem]{Proposition}

\newcommand{\bnum}{\begin{enumerate}}
\newcommand{\enum}{\end{enumerate}}

\numberwithin{equation}{section}

\begin{document}

\title{On commuting probability of finite rings II}
\author{Parama Dutta and Rajat Kanti Nath\footnote{Corresponding author}}
\date{}
\maketitle
\begin{center}\small{\it
Department of Mathematical Sciences, Tezpur University,\\ Napaam-784028, Sonitpur, Assam, India.\\



Emails:\, parama@gonitsora.com and rajatkantinath@yahoo.com}
\end{center}

\medskip

\begin{abstract}
The aim of this paper is to study  the probability that the commutator of an arbitrarily chosen pair of elements, each from two different subrings of a finite non-commutative ring equals a given element of that ring. We obtain several results on this probability including a computing formula, some bounds and characterizations.
\end{abstract}

\medskip

\noindent {\small{\textit{Key words:}  finite ring, commuting probability, ${\mathbb{Z}}$-isoclinism of rings.}}

\noindent {\small{\textit{2010 Mathematics Subject Classification:}
16U70, 16U80.}}

\medskip

\section{Introduction}
Throughout this paper $R$  denotes a finite non-commutative ring. The commuting probability of $R$, denoted by $\Pr(R)$, is the probability that a randomly chosen pair of elements of $R$ commute. That is
\[
\Pr(R) = \frac{|\lbrace(r, s)\in R\times R : [r, s] = 0 \rbrace|}{|R\times R|}
\]
where $[r, s] := rs - sr$ is the additive commutator of $r$ and $s$ and $0$ is the zero element of $R$. Clearly $\Pr(R) =  1$ if and only if $R$ is commutative. The study of commuting probability of finite rings  was initiated by MacHale \cite{dmachale} in the year 1976 motivated by the commuting probability of finite groups. After Erd$\ddot{\rm o}s$ \cite{pEpT68}, many authors have worked on  the commuting probability of finite groups and its generalizations (conf. \cite{Dnp13} and the references therein) but somehow people have neglected  commuting probability of finite rings. At this moment, we have very few papers in the literature on  commuting probability of finite rings \cite{BM,BMS,jutireka,jutirekha2,dmachale}. In this paper, we study a generalization of $\Pr(R)$.

Let  $S$ and $K$ be two subrings of $R$ and $r \in R$. We define ${\Pr}_r(S, K)$ in the following way
\begin{equation}\label{mainformula}
{\Pr}_r(S, K) = \frac{|\lbrace(s, k)\in S\times K : [s,k] = r\rbrace|} {|S\times K|}.
\end{equation}
Thus ${\Pr}_r(S, K)$ is the probability that the additive commutator of a randomly chosen pair of elements, one from $S$ and the other from $K$, is equal to a given element $r$ of $R$. This generalizes $\Pr(R)$ since ${\Pr}_r(S, K) = \Pr(R)$ if $S = K = R$ and $r = 0$. If $r = 0$  then
\[
{\Pr}_r(S, K) = \Pr(S, K) = \frac{|\lbrace(s, k) \in S\times K : sk = ks\rbrace|} {|S\times K|}.
\]
It may be mentioned here that some connections between $\Pr(S, K)$ and generalized non-commuting graph of $R$ can be found in \cite{jutirekha2}. In \cite{jutireka}, $\Pr(S, R)$ is studied extensively.

In this paper, we obtain several results on ${\Pr}_r(S, K)$   including a computing formula, some bounds and characterizations.  The motivation of this paper lies in \cite{DN10, nY15,PS08} where analogous generalizations of  commuting probability of finite groups are studied.

For any two subrings $S$ and $K$, we write $[S, K]$ and $[s,K]$ for $s\in S$  to denote the additive subgroups of $(R, +)$ generated by the sets $\lbrace [s, k] : s\in S,k\in K\rbrace$ and $\lbrace [s, k] : k\in K\rbrace$ respectively. It can be seen that any element of $[s,K]$ is of the form $[s, k]$ for some $k \in K$ and so $[s,K] = \{[s, k] : k\in K\}$.
 Let  $Z(S, K) := \{s \in S : sk = ks\,  \forall k \in K\}$. Then   $Z(S, K) = S\cap Z(K)$ if $S\subseteq K$ and $Z(K, K) = Z(K)$, the center of $K$. Also, for $r \in R$ the set $C_S(r) := \{s\in S : sr = rs\}$   is a subring of $S$ and $\underset{r \in K}{\cap} C_S(r) = Z(S, K)$. We write   $R/S$ or $\frac{R}{S}$  to denote the additive factor group, for any subring $S$ of $R$,  and $|R : S|$ to denote the index of    $(S, +)$ in   $(R, +)$.     The isomorphisms considered in this paper are the additive group isomorphisms.
 It is easy to see that
${\Pr}_r(S, K) = 1$ if and only if $r = 0$  and $[S,K] = \{0\}$.
Also,
${\Pr}_r(S, K) = 0$  if and only if $r\notin \{[s,k]:s\in S,k\in K\}$.
Therefore, we consider $r$ to be an element of $\{[s,k]:s\in S,k\in K\}$ throughout the paper.

\section{Preliminary results}


In this section, we deduce some elementary results on ${\Pr}_r(S, K)$ and derive a computing formula for ${\Pr}_r(S, K)$. We begin with the following result which shows that ${\Pr}_r(S, K)$ is not symmetric with respect to $S$ and $K$.


\begin{proposition}\label{symmetricity}
Let $S$ and $K$ be two subrings of $R$ and $r \in [S, K]$. Then
${\Pr}_r(S, K) = {\Pr}_{-r}(K, S)$. However, if $2r = 0$  then ${\Pr}_r(S, K) = {\Pr}_{r}(K, S)$.
\end{proposition}

\begin{proof}
Let $X = \{(s, k) \in S \times K : [s, k] = r\}$ and $Y = \{(k, s) \in K \times S : [k, s] = -r\}$. It is easy to see that $(s, k) \mapsto (k, s)$ defines a bijective mapping from $X$ to $Y$. Therefore, $|X| = |Y|$ and the result follows.


Second part follows from the fact that $r = -r$ if $2r = 0$.
\end{proof}


\begin{proposition}
Let $S_i$ and $K_i$ be two subrings of finite non-commutative rings $R_i$ for $i = 1, 2$ respectively. If $(r_1, r_2) \in R_1 \times R_2$  then
\[
{\Pr}_{(r_1, r_2)}(S_1 \times S_2, K_1\times K_2) = {\Pr}_{r_1}(S_1, K_1){\Pr}_{r_2}(S_2, K_2).
\]
\end{proposition}

\begin{proof}
Let $X_i = \{(s_i, k_i) \in S_i\times K_i : [s_i, k_i] = r_i\}$ for $i = 1, 2$ and $Y = \{((s_1, s_2), (k_1, k_2)) \in (S_1\times S_2) \times (K_1\times K_2) : [(s_1, s_2),(k_1, k_2)]= (r_1, r_2)\}$. Then $((s_1, k_1), (s_2, k_2)) \mapsto ((s_1, s_2),(k_1, k_2))$ defines a bijective map from $X_1 \times X_2$ to $Y$. Therefore, $|Y| = |X_1||X_2|$ and hence the result follows.
\end{proof}

Initially, it was challenging for us to derive a computing formula for ${\Pr}_r(S, K)$ since there is no analogous concept of conjugacy class and no analogous character theoretic results for rings. Finally, we are able to get a formula. The following two lemmas play important role in obtaining our computing formula for ${\Pr}_r(S, K)$.

\begin{lemma}\label{lemma1}
Let $K$ be any subring of $R$. If  $x \in R$ then
$
|[x, K]|=\frac {|K|}{|C_K(x)|}.
$
\end{lemma}

\begin{proof}
Note that  $[x, k] \mapsto k + C_K(x)$ defines an isomorphism from $[x, K]$ to $\frac {K}{C_K(x)}$. Hence, the lemma follows.
\end{proof}

\begin{lemma}\label{lemma2}
Let $S$ and $K$ be two subrings of $R$ and $T_{s, r} = \{k\in K : [s, k] = r\}$ where $s\in S$ and $r\in R$. Then we have the followings 
\begin{enumerate}
\item If $T_{s, r}\neq \phi$ then $T_{s, r} = t + C_K(s)$ for some $t\in T_{s, r}$.
\item  $T_{s, r} \ne \phi$ if and only if $r \in [s, K]$.
\end{enumerate}
\end{lemma}

\begin{proof}
Let $t \in T_{s, r}$ and  $p \in t + C_K(s)$. Then $[s, p] = r$ and so $p \in T_{s, r}$. Therefore,  $t + C_K(s) \subseteq T_{s, r}$. Again, if $k \in T_{s, r}$ then  $(k - t) \in C_K(s)$ and so $k \in t + C_K(s)$. Therefore, $t + C_K(s)\subseteq T_{s, r}$. Hence part (a) follows.

Part (b) follows from the fact that $y \in T_{s, r}$ if and only of $r \in [s, K]$.
\end{proof}

\noindent Now we state and prove our main result of this section.

\begin{theorem}\label{com-thm}
Let $S$  and $K$ be two subrings of $R$. Then
\[
{\Pr}_r(S,K) = \frac {1}{|S||K|}\underset{r\in [s, K]}{\underset{s\in S}{\sum}}|C_K(s)| = \frac {1}{|S|}\underset{r\in [s, K]}{\underset{s\in S}{\sum}}\frac{1}{|[s, K]|}.
\]

\end{theorem}

\begin{proof}
Note that $\{(s, k) \in S\times K : [s, k] = r\} = \underset{s\in S}{\cup}(\{s\}\times T_{s, r})$. Therefore, by \eqref{mainformula} and Lemma \ref{lemma2}, we have
\begin{equation} \label{comfor1}
|S||K|{\Pr}_r(S,K) = \underset{s \in S}{\sum} |T_{s, r}| = \underset{r\in [s, K]}{\underset{s\in S}{\sum}}|C_K(s)|.
\end{equation}

The second part follows  from \eqref{comfor1} and Lemma  \ref{lemma1}.
\end{proof}

Using Proposition \ref{symmetricity}, we get the following corollary of Theorem \ref{com-thm}.
\begin{corollary}\label{formula1}
Let $S$  and $K$ be two subrings of $R$. Then
\[
{\Pr}(K, S) = {\Pr}(S, K) = \frac {1}{|S||K|}\sum_{s\in S}|C_K(s)|  = \frac {1}{|S|}\sum_{s\in S}\frac{1}{|[s, K]|}.
\]
\end{corollary}
It is worth mentioning that Equation (2.1) of \cite{BMS} and \cite{jutireka} also follow from Corollary \ref{formula1}.

We conclude this section by the following two  lemmas.
\begin{lemma}\label{lemma02}
{\rm \cite[Lemma 2.12]{jutireka}} Let $H$ and $N$ be two subrings of a   non-commutative ring $R$ such that $N$ is an ideal of $R$ and $N \subseteq H$.
Then
\[
\frac{C_H(x) + N}{N} \subseteq C_{H/N}(x + N) \; \text{for all} \; x \in R,
\]
where $H/N$ is a factor ring. The equality holds if $N \cap [H, R] = \{0\}$.
\end{lemma}

\begin{lemma}\label{lemma002}
Let $S \subseteq K$ be two subrings of $R$. If $S$ is non-commutative then $\frac{S}{Z(S, K)}$ is not cyclic.
\end{lemma}


\section{Some bounds and characterizations}
In this section, we derive some bounds for ${\Pr}_r(S,K)$ and characterizations of subrings $S$ and $K$ in terms of $\Pr(S, K)$. We begin with the following lower bounds.

\begin{proposition}
Let $S$  and $K$ be two subrings of $R$. If $r \ne 0$ then
\begin{enumerate}
\item ${\Pr}_r(S,K)\geq \frac{|Z(S,K)||Z(K,S)|}{|S||K|}$.
\item If $S \subseteq K$ then ${\Pr}_r(S, K)\geq \frac{2|Z(S, K)||Z(K, S)|}{|S||K|}$.
\item ${\Pr}_r(R)\geq \frac{3}{|R : Z(R)|^2}$.
\end{enumerate}
\end{proposition}

\begin{proof}
Since $r \ne 0$ we have the set ${\mathcal{C}} := \{(x, y) \in S \times K : [x, y] = r\}$ is non empty. Let $(s, k) \in {\mathcal{C}}$ then $(s, k) \notin Z(S, K)\times Z(K, S)$, otherwise $[s, k] = 0$. Now, for part (a) it is sufficient to note that the coset $(s, k) + \left(Z(S, K)\times Z(K, S)\right)$ is a subset of ${\mathcal{C}}$. If $S \subseteq K$ then $(s, k) + \left(Z(S, K)\times Z(K, S)\right)$ and $(s, k + s) + \left(Z(S, K)\times Z(K, S)\right)$ are two disjoint subsets of ${\mathcal{C}}$. Therefore, part (b) follows. For part (c), we consider $S = K = R$ and notice that $(s, k) + \left(Z(R)\times Z(R)\right)$, $(s + k, k) + \left(Z(R)\times Z(R)\right)$ and $(s, k + s) + \left(Z(R)\times Z(R)\right)$ are three disjoint subsets of $\{(x, y) \in R \times R : [x, y] = r\}$.
\end{proof}

\begin{proposition}\label{ub02}
Let $S$  and $K$ be two subrings of $R$. Then ${\Pr}_r(S, K) \leq {\Pr}(S, K)$ with equality if and only if $r = 0$. Further, if $S \subseteq K$ then ${\Pr}(S, K) \leq |K : S|\Pr(K)$ with equality if and only if $S = K$.
\end{proposition}
\begin{proof}
By Theorem \ref{com-thm} and Corollary \ref{formula1}, we have
\[
{\Pr}_r(S, K) = \frac {1}{|S||K|}\underset{r \in [s, K]}{\underset{s\in S}{\sum}}|C_K(s)|
 \leq \frac {1}{|S||K|}\underset{s\in S}{\sum}|C_K(s)|
=  \Pr(S, K).
\]
The equality holds if and only if $r = 0$.

If $S \subseteq K$ then we have
\[
\Pr(S, K) = \frac {1}{|S||K|}\underset{s\in S}{\sum}|C_K(s)| \leq \frac {1}{|S||K|}\underset{s\in K}{\sum}|C_K(s)| = |K : S|\Pr(K).
\]
The equality holds if and only if $S = K$.
\end{proof}




\begin{proposition}
Let $S$  and $K$ be two subrings of $R$. If $p$ is the smallest prime dividing $|R|$ and $r \ne 0$ then
\[
{\Pr}_r(S, K)\leq \frac {|S| - |Z(S, K)|}{p|S|} < \frac {1}{p}.
\]
\end{proposition}
\begin{proof}
Since  $r \ne 0$ we have $S \ne Z(S, K)$. If $s \in Z(S, K)$ then $r \notin [s, K]$. If $s \in S \setminus Z(S, K)$ then $C_K(s) \ne K$. Therefore, by Lemma \ref{lemma1}, we have  $|[s, K]| = |K : C_K(s)| > 1$. Since   $p$ is the smallest prime dividing $|R|$ we have $|[s, K]| \geq p$.  Hence the result follows from Theorem \ref{com-thm}.
\end{proof}

\begin{proposition}\label{ub03}
If $S_1\subseteq S_2$ and $K_1\subseteq K_2$ are subrings of $R$. Then
\[
{\Pr}_r(S_1, K_1) \leq |S_2 : S_1||K_2 : K_1|{\Pr}_r(S_2, K_2).
\]
\end{proposition}
\begin{proof}
By Theorem \ref{com-thm}, we have
\begin{align*}
|S_1||K_1|{\Pr}_r(S_1, K_1) = &\underset{r  \in [s, K_1]}{\underset{s \in S_1}{\sum}}|C_{K_1}(s)|\\
\leq &\underset{r \in [s, K_2]}{\underset{s\in S_2}{\sum}}|C_{K_2}(s)| = |S_2||K_2|{\Pr}_r(S_2,K_2).
\end{align*}
Hence the result follows.
\end{proof}
\noindent Note that  equality holds in Proposition \ref{ub03} if and only if
$r \notin  [s, K_2]$   for all   $s\in S_2 \setminus S_1$,   $r \notin  [s, K_2] \setminus [s, K_1]$   for all $s\in S_1$
   and    $C_{K_1}(s) =   C_{K_2}(s$)   for all $s\in S_1$   with $r \in [s, K_1]$.
If $r = 0$ then the condition of equality reduces to $S_1 = S_2$ and $K_1 = K_2$.

\begin{corollary}
If $S_1\subseteq S_2$  are two subrings of $R$. Then
\[
{\Pr}_r(S_1, R) \leq |S_2 : S_1|{\Pr}(S_2, R).
\]
The equality holds if and only if $r = 0$ and $S_1 = S_2$.
\end{corollary}
\begin{proof}
Putting $K_1 = K_2 = R$ in Proposition \ref{ub03}, we get
\[
{\Pr}_r(S_1, R) \leq |S_2 : S_1|{\Pr}_r(S_2, R).
\]
Hence the result follows from Proposition \ref{ub02}.
\end{proof}

\begin{proposition}\label{boundS_1K_1}
Let $S, K_1$ and $K_2$ be three subrings of $R$. If  $K_1\subseteq K_2$  then
\[
{\Pr}(S, K_1) \geq {\Pr}(S, K_2)\geq \frac {1}{|K_2 : K_1|}\left(\Pr(S, K_1) + \frac {|K_2| - |K_1|}{|S||K_1|}\right).
\]
The first equality holds if and only if $[s, K_1] = [s, K_2]$ for all $s \in S$ and the second  equality holds if and only if $C_S(k) = \{0\}$ for all $k \in K_2\setminus K_1$.
\end{proposition}
\begin{proof}
Since $[s, K_1] \subseteq [s, K_2]$ for all $s \in S$, by Corollary \ref{formula1}, we have
\[
{\Pr}(S, K_1) = \frac{1}{|S|} \sum_{s \in S}\frac{1}{|[s, K_1]|} \geq \frac{1}{|H|} \sum_{s \in S}\frac{1}{|[s, K_2]|} = {\Pr}(S, K_2)
\]
with equality  if and only if $[s, K_1] = [s, K_2]$ for all $s \in S$.

By  Corollary \ref{formula1}, we also have
\begin{align}\label{inequality-1}
\Pr(S, K_2) =  \Pr(K_2, S)\nonumber
= & \frac {1}{|S||K_2|}\underset{k \in K_2}{\sum}|C_S(k)|\nonumber\\
= &\frac {\Pr(S,K_1)}{|K_2 : K_1|}  + \frac {1}{|S||K_2|}\underset{k \in K_2\setminus K_1}{\sum}|C_S(k)|.
\end{align}
We have $|C_S(k)| \geq 1$ for all $k \in K_2\setminus K_1$. Therefore,
\begin{equation}\label{inequality-2}
\underset{k \in K_2\setminus K_1}{\sum}|C_S(k)| \geq |K_2| - |K_1|
\end{equation}
the equality holds if and only if $|C_S(k)| = 1$ for all $k\in K_2\setminus K_1$. Hence the result follows from \eqref{inequality-1} and \eqref{inequality-2}.
\end{proof}
\noindent It is worth mentioning that the second part of  \cite[Theorem 2.4]{jutireka} follows from the first part of Proposition \ref{boundS_1K_1}.
\begin{proposition}
Let $S \subseteq K$ be two subrings of $R$. If $p$ is the smallest prime dividing  $|R|$  and $|S : Z(S, K)| = p^n$   then $\Pr(S, K)\leq \frac {p^n + p - 1}{p^{n + 1}}$. Moreover, if $S = K$ then we have $\Pr(S, K) \geq  \frac {p^n + p^{n - 1} - 1}{p^{2n - 1}}$.
\end{proposition}
\begin{proof}
If $s\in S\setminus Z(S, K)$ then $C_K(s)\neq K$ and hence $\frac {|K|}{|C_K(s)|}\geq p$. Therefore, by Corollary \ref{formula1},  we have
\begin{align*}
\Pr(S, K) = &\frac {1}{|S||K|}\underset{s\in Z(S, K)}{\sum}|C_K(s)| + \frac {1}{|S||K|}\underset{s\in S\setminus Z(S, K)}{\sum}|C_K(s)|\\
\leq & \frac{|Z(S, K)|}{|S|} + \frac{|S|-|Z(S, K)|}{p|S|} = \frac {p^n + p - 1}{p^{n + 1}}.
\end{align*}
If $S = K$ then for $s \in S \setminus Z(S, K)$ we have $Z(S, K) \subsetneq C_K(s) \ne K$ and so $|C_K(s)| \geq p|Z(S, K)|$. Therefore, by Corollary \ref{formula1}, we have
\[
\Pr(S, K) \geq \frac{|Z(S, K)|}{|S|} + \frac{p|Z(S, K)|(|S|-|Z(S, K)|)}{|S||K|} = \frac {p^n + p^{n - 1} - 1}{p^{2n - 1}}.
\]
\end{proof}

\begin{theorem}\label{prop1}
Let $S$  and $K$ be two subrings of $R$ and $p$  the smallest prime dividing $|R|$. Then
\begin{align*}
\Pr(S, K)\geq & \frac {|Z(S, K)|}{|S|}+\frac {p(|S|-|X_S|-|Z(S, K)|)+|X_S|}{|S||K|}\\
\textup{ and } \Pr(S,K)\leq &\frac {(p-1)|Z(S, K)|+|S|}{p|S|}-\frac {|X_S|(|K|-p)}{p|S||K|}
\end{align*}
where $X_S = \{s\in S: C_K(s) = \{0\}\}$. Moreover, in each of these bounds, $S$ and $K$ can be interchanged.
\end{theorem}

\begin{proof}
If $[S, K] = \{0\}$ then $\Pr(S, K) = 1$, $Z(S, K) = S$ and $X_S = \{0\}$ or $\phi$ according as $K = \{0\}$ or $K \ne \{0\}$. If $K \ne \{0\}$  then both the sides of the above inequalities give  $1$. Otherwise, it is routine to see that
\[
1 - \frac{p - 1}{|S|} < \Pr(S, K) < 1 + \frac{p - 1}{p|S|}.
\]

 Let $[S, K]\neq \{0\}$ then $X_S\cap Z(S, K) = \phi$. Therefore
\begin{align}\label{eq0002}
\underset{s\in S}{\sum}|C_K(s)|=&\underset{s\in X_S}{\sum}|C_K(s)|+\underset{s\in Z(S,K)}{\sum}|C_K(s)|+\underset{s\in S\setminus (X_S\cup Z(S,K))}{\sum}|C_K(s)|\nonumber\\
=& |X_S| +  |K| |Z(S, K)| + \underset{s\in S\setminus (X_S\cup Z(S,K))}{\sum}|C_K(s)|.
\end{align}
Notice that for all $s\in S\setminus (X_S\cup Z(S,K))$ we have $\{0\}\neq C_K(s)\neq K$ which gives $p\leq |C_K(s)|\leq \frac {|K|}{p}$. Hence
\begin{align}\label{eq0003}
(|S| - |X_S| - |Z(S,K)|)p \leq &\underset{s\in S\setminus (X_S\cup Z(S,K))}{\sum}|C_K(s)|\nonumber\\
\leq & (|S| - |X_S| - |Z(S,K)|)\frac{|K|}{p}.
\end{align}
Now the required inequalities can be obtained using Corollary \ref{formula1}, \eqref{eq0002} and \eqref{eq0003}.

The last part of the proposition follows from the fact that $\Pr(S, K) = \Pr(K, S)$.
\end{proof}

Putting $K = R$ in Theorem \ref{prop1} we get an upper bound for $\Pr(S, R)$ which is better than the upper bound obtained in  \cite[Theorem 2.5]{jutireka}.
\begin{corollary}
Let $S$  and $K$ be two subrings of $R$. If $[S, K]\neq \{0\}$ and $p$  the smallest prime dividing $|R|$  then $\Pr(S, K)\leq \frac {2p-1}{p^2}$. In particular, $\Pr(S,K)\leq \frac {3}{4}$.
\end{corollary}
\begin{proof}
Since $[S, K]\neq \{0\}$ we have  $Z(S,K) \ne S$. Therefore $|Z(S,K)|\leq \frac {|S|}{p}$. Hence, by second part of Theorem \ref{prop1}, we have
\[
\Pr(S, K)\leq \frac {(p - 1)|Z(S, K)|+|S|}{p|S|} \leq \frac {(p - 1)\frac {|S|}{p} + |S|}{p|S|} = \frac {2p-1}{p^2}.
\]
The particular case follows from the fact  $p \geq 2$ and $\frac{2p - 1}{p^2} \leq \frac{3}{4}$ for any prime~$p$.
\end{proof}

We have, for all $s \in S$
\begin{equation}\label{eqlb}
|[S, K]|  \geq |[s, K]| = |K : C_K(s)|.
\end{equation}
Therefore, by Corollary \ref{formula1} and \eqref{eqlb}, we have the following lower bound for   $\Pr(S, K)$.
\begin{proposition}\label{newlb1}
Let $S$ and $K$ be two subrings of  $R$. Then
\[
\Pr(S, K) \geq \frac{1}{|[S, K]|}\left(1 + \frac{|[S, K]| - 1}{|S : Z(S, K)|} \right).
\]
In particular, if $Z(S, K) \ne S$ then $\Pr(S, K) > \frac{1}{|[S, K]|}$.
\end{proposition}
\noindent It is worth noting that \cite[Theorem 2. 17]{jutireka} follows from Proposition \ref{newlb1}.

\begin{theorem}\label{prop2}
Let $S$  and $K$ be two subrings of $R$ such that $\Pr(S, K) = \frac{2p - 1}{p^2}$ for some prime $p$. Then $p$ divides $|R|$. If $p$ is the smallest prime dividing $|R|$ then
\[
\frac{S}{Z(S,K)}\cong\mathbb Z_p\cong\frac{K}{Z(K,S)}
\]
and hence $S\neq K$. In particular, if $\Pr(S, K) = \frac{3}{4}$ then
\[
\frac{S}{Z(S,K)}\cong\mathbb Z_2\cong\frac{K}{Z(K,S)}.
\]
\end{theorem}
\begin{proof}
If $\Pr(S,K) = \frac{2p-1}{p^2}$ then, by Corollary \ref{formula1}, we have   $p$ divides $|S||K|$ and hence $p$ divides $|R|$.

For the second part we have, by Theorem \ref{prop1},
\[
\frac{2p-1}{p^2} \leq \frac {(p - 1)|Z(S, K)|+|S|}{p|S|} = \frac{p - 1}{p|S : Z(S, K)|} + \frac{1}{p}
\]
which gives $|S : Z(S, K)|\leq p$. Since  $\Pr(S, K) \ne 1$ we have $S \ne Z(S, K)$ and hence $|S : Z(S, K)| = p$. Therefore, $\frac{S}{Z(S,K)}\cong\mathbb Z_p$. Interchanging the role of $S$ and $K$ we get $\frac{K}{Z(K, S)}\cong\mathbb Z_p$.

If $S = K$ then $\frac{S}{Z(S,K)}\cong\mathbb Z_p\cong\frac{K}{Z(K,S)}$ implies $S$ and $K$ are commutative (by Lemma \ref{lemma002}). Therefore, $\frac{S}{Z(S,K)}$ and $\frac{K}{Z(K,S)}$ are trivial group, which is a contradiction. Hence, $S \ne K$.

The last part follows considering   $p = 2$.
\end{proof}
\noindent The following lemma  is useful in the subsequent results.
\begin{lemma}\label{newlemma}
Let $S$ and $K$ be two subrings of $R$. If $[x, S] \subseteq [x, K]$ for all $x \in S\cup K$ then
\[
\Pr(K) \leq \Pr(S,K) \leq \Pr(S).
\]
\end{lemma}
\begin{proof}
By Corollary \ref{formula1} we have
\[
\Pr(S) = \frac {1}{|S|}\underset{s\in S}{\sum}\frac{|C_S(s)|}{|S|} \geq \frac {1}{|S|}\underset{s\in S}{\sum}\frac {|C_K(s)|}{|K|} = \Pr(S, K)
\]
and
\[
\Pr(S, K) = \frac {1}{|K|}\underset{k \in K}{\sum}\frac{|C_S(k)|}{|S|} \geq \frac {1}{|K|}\underset{k \in K}{\sum}\frac {|C_K(k)|}{|K|} = \Pr(K).
\]
Hence the lemma follows.
\end{proof}
It may be mentioned here that  \cite[Theorem 2.2]{jutireka} follows from the above lemma.

\begin{proposition}
Let $S$ and $K$ be two subrings of  $R$ such that $[x, S] \subseteq [x, K]$ for all $x \in S$. If $S$ is non-commutative and $p$ is the smallest prime dividing $|S|$  then $\Pr(S,K)\leq \frac {p^2 + p - 1}{p^3}$.
\end{proposition}

\begin{proof}
The result follows from  \cite[Theorem 2]{dmachale} and Lemma \ref{newlemma}.
\end{proof}

\begin{theorem}\label{5/8like}
Let $S \subseteq K$ be two non-commutative subrings of   $R$ and $\Pr(S, K) = \frac {p^2 + p - 1}{p^3}$ for some prime $p$. Then $p$ divides $|R|$. If $p$ is the smallest prime dividing $|R|$ then
\[
\frac {S}{Z(S, K)}\cong \mathbb Z_p\times \mathbb Z_p.
\]
In particular, if $\Pr(S, K) = \frac {5}{8}$ then $\frac {S}{Z(S, K)}\cong \mathbb Z_2\times \mathbb Z_2$.
\end{theorem}

\begin{proof}
If $\Pr(S, K) = \frac {p^2 + p - 1}{p^3}$ then by Corollary \ref{formula1}, we have $p$ divides $|S||K|$ and hence $p$ divides $|R|$.

By second part of Theorem \ref{prop1}, we have
\[
 \frac {p^2 + p - 1}{p^3} \leq \frac{(p - 1)|Z(S, K)| + |S|}{p|S|} = \frac{p - 1}{p|S : Z(S, K)|} + \frac{1}{p}
\]
which gives $|S : Z(S, K)| \leq p^2$. Since $\Pr(S, K) \ne 1$ we have $S \ne Z(S, K)$. Also $\frac{S}{Z(S, K)}$ is not cyclic as $S$ is non-commutative (by Lemma \ref{lemma002}). Hence $\frac{S}{Z(S, K)} \cong \mathbb Z_p\times \mathbb Z_p$.

The particular case follows considering $p = 2$.
\end{proof}
The following proposition gives partial converse of Theorem \ref{prop2} and Theorem \ref{5/8like}.
\begin{theorem}
Let $S \subseteq K$ be two subrings of  $R$.
\begin{enumerate}
\item
If $\frac {S}{Z(S, K)}\cong \mathbb Z_p$ and $|K : S| = n$ then $\Pr(S, K)\geq \frac {n + p - 1}{np}$. Further, if $p$ is the smallest prime dividing $|R|$ and $|K : S| = p$ then $\Pr(S, K) = \frac {2p - 1}{p^2}$.
\item
If $\frac {S}{Z(S, K)}\cong \mathbb Z_p\times \mathbb Z_p$ and $|K : S| = n$ then $\Pr(S, K)\geq \frac {(n + 2)p^2 - 2}{np^4}$. Further,  if $p$ is the smallest prime dividing $|R|$ and $|K : S| = 1$  then $\Pr(S, K) = \frac {p^2 + p - 1}{p^3}$.
\end{enumerate}
\end{theorem}
\begin{proof}
(a) Since $\frac {S}{Z(S, K)}$ is cyclic, we have $S$ is commutative (by Lemma \ref{lemma002}). Therefore, if $s\in S\setminus Z(S, K)$ then $|C_K(s)|\geq |S| = \frac {|K|}{n}$. Now, by Corollary \ref{formula1}, we have
\begin{align*}
\Pr(S, K) &= \frac {|Z(S,K)|}{|S|} + \frac {1}{|S||K|}\underset{s\in S\setminus Z(S, K)}{\sum}|C_K(s)|\\
&\geq \frac {1}{p} + \frac {|S|-|Z(S,K)|}{n|S|} = \frac{n + p - 1}{np}.
\end{align*}
If $p$ is the smallest prime dividing $|R|$ and $|K : S| = p$ then $|C_K(s)| = \frac {|K|}{p}$  for $s\in S\setminus Z(S, K)$   and hence $\Pr(S, K) = \frac {2p - 1}{p^2}$.

(b) We have $Z(S, K) \subsetneq C_K(s)$ if $s\in S\setminus Z(S, K)$. Also, $|Z(S,K)|$ divides $|C_K(s)|$ and so
\begin{equation}\label{eqconv1}
|C_K(s)| \geq 2|Z(S,K)| = \frac{2|S|}{p^2} = \frac{2|K|}{np^2}
\end{equation}
for all $s\in S\setminus Z(S, K)$.
Now, by Corollary \ref{formula1} and \eqref{eqconv1}, we have
\[
\Pr(S, K) \geq \frac{1}{p^2} + \frac{2(|S| - |Z(S, K)|)}{np^2|S|} = \frac {(n + 2)p^2 - 2}{np^4}.
\]
If $p$ is the smallest prime dividing $|R|$ and $|K : S| = 1$ then $|C_K(s)| \geq p|Z(S,K)| =  \frac{|S|}{p}$ for $s\in S\setminus Z(S, K)$. Also, $C_K(s) \subsetneq S$ and so $|C_K(s)| =  \frac{|S|}{p}$ for  $s\in S\setminus Z(S, K)$. Therefore, Corollary \ref{formula1} gives
\[
\Pr(S, K) = \frac{1}{p^2} + \frac{|S| - |Z(S, K)|}{p|S|} = \frac {p^2 + p - 1}{p^3}.
\]
\end{proof}

\noindent The following corollary follows immediately.
\begin{corollary}
Let $S \subseteq K$ be two subrings of  $R$. Then
\begin{enumerate}
\item
If $\frac {S}{Z(S, K)}\cong \mathbb Z_2$ and $|K : S| = 2$ then $\Pr(S, K)= \frac {3}{4}$.
\item
If $\frac {S}{Z(S, K)}\cong \mathbb Z_2\times \mathbb Z_2$ and $|K : S| = 1$ then $\Pr(S, K)=\frac{5}{8}$.
\end{enumerate}
\end{corollary}
\noindent We also have the following result.
\begin{proposition}
Let $S$ and $K$ be two subrings of $R$ such that $\frac {S}{Z(S, K)}\cong \mathbb Z_p\times \mathbb Z_p$. If $p$ is the smallest prime dividing $|R|$ and $|[s,K]| = p$ for all $s \in S \setminus Z(S, K)$ then $\Pr(S, K) = \frac {p^2 + p - 1}{p^3}$.
\end{proposition}
\begin{proof}
If $p$ is the smallest prime dividing $|R|$ and $|[s,K]| = p$ for all $s \in S \setminus Z(S, K)$ then by Corollary \ref{formula1}, we have
\[
\Pr(S, K) = \frac{|Z(S,K)|}{|S|} + \frac{1}{|S|} \underset{s\in S\setminus Z(S, K)}{\sum}\frac{1}{p} = \frac {p^2 + p - 1}{p^3}.
\]
\end{proof}

We shall conclude this section by the following proposition which is an improvement of \cite[Theorem 2.13]{jutireka}.

\begin{proposition}
Let $S$ and $K$ be two subrings  of $R$ and $I$ be an ideal of $R$ such that $I\subseteq S \cap K$. Then $\Pr(S, K)\leq \Pr\left(\frac {S}{I},\frac {K}{I}\right)\Pr(I)$ where $\frac {S}{I}$ and $\frac {K}{I}$ are factor rings. The equality holds if $I \cap [S, R] = \{0\}$.
\end{proposition}

\begin{proof}
 By Corollary \ref{formula1} and Lemma \ref{lemma02}, we have
\begin{align}\label{factor-eq1}
|S||K| \Pr(S,K)
= &\underset{P \in \frac{K}{I}}{\sum}\underset{k \in P}{\sum}\frac{|C_S(k)|}{|I \cap C_S(k)|} |C_I(k)| \nonumber\\
= &\underset{P \in \frac{K}{I}}{\sum}\underset{k \in P}{\sum}\left|\frac{C_S(k) + I}{I}\right| |C_I(k)| \nonumber\\
\leq &\underset{P \in \frac{K}{I}}{\sum}\underset{k \in P}{\sum}|C_{\frac{S}{I}}(k+I)||C_I(k)|\nonumber\\
= &\underset{P\in \frac{K}{I}}{\sum} |C_{\frac{S}{I}}(P)| \underset{u \in I}{\sum}|C_K(u) \cap P|.
\end{align}
Let $a + I = P$ where $a\in K\setminus I$. If $C_K(u)\cap P = \phi$  then $|C_K(u)\cap P| = 0$.   Therefore, $|C_K(u)\cap P|<|C_I(u)|$ as $|C_I(u)|\geq 1$. On the other hand, if $C_K(u)\cap P\neq \phi$ then there exist $x\in C_K(u)\cap P$ and $x = a + v$ for some  $v\in I$ which implies $x +I  = a + I = P$. Therefore,
\begin{align*}
C_K(u)\cap P
=(x + I)\cap (x + C_K(u))
= x + (I\cap C_K(u))
= x + C_I(u)
\end{align*}
and so $|C_K(u)\cap P| = |C_I(u)|$. Therefore, in both the cases, $|C_K(u)\cap P|\leq|C_I(u)|$ and so, by \eqref{factor-eq1}, we have
\begin{align*}
|S||K| \Pr(S,K)
\leq &\underset{S \in \frac{K}{I}}{\sum} |C_{\frac{S}{I}}(P)| \underset{u\in I}{\sum}|C_I(u)|\\
=&\left|\frac{S}{I} \right| \left|\frac{K}{I} \right| \Pr \left(\frac{S}{I},\frac{K}{I}\right)|I|^2\Pr(I)\\
=&|S||K|\Pr \left(\frac{S}{I},\frac{K}{I}\right)\Pr(I).
\end{align*}
Thus, first part of the result follows.

If $I\cap [S, R] = 0$ then equality holds in Lemma \ref{lemma02} and hence equality holds in \eqref{factor-eq1}. Further, if $P = a + I$ and $k\in P$ then $k = a + u$ for some $u\in I$ and $a\in K\setminus I$. Therefore, $k\in K$ and $u\in I\subseteq S$ and hence $[u, k]\in [S, K]$.  Also $u\in I$ and  $k \in K\subseteq R$ gives $uk, ku\in I$ and so $[u, k]\in I.$ Hence, $[u, k]\in [S, R]\cap I$ and so $k\in C_K(u)$. Therefore, $C_K(u)\cap P\neq \phi$ and  $|C_K(u)\cap P|=|C_I(u)|$. Hence the equality holds.
\end{proof}

\section{Isoclinism  and commuting probability}

The concept of isoclinism between groups was  introduced by Hall \cite{pH40} in 1940.  In 1995, Lescot \cite{pL95} showed that the commuting probabilities of two isoclinic finite groups are same. Recently,  Bukley, MacHale and Ni Sh$\acute{\textup{e}}$\cite{BMS} have introduced  $\mathbb Z$-isoclinism between two rings and showed that the commuting probabilities of two $\mathbb Z$-isoclinic finite rings are same. Further, Dutta, Basnet  and Nath  \cite{jutirekha2} have generalized the  concept of $\mathbb Z$-isoclinism between two rings as given in the following definition.
\begin{definition}
Let $R_1$ and $R_2$ be two rings with subrings $S_1, K_1$ and $S_2, K_2$ respectively such that $S_1\subseteq K_1$ and $S_2\subseteq K_2$. A pair of rings $(S_1,K_1)$ is said to be $\mathbb Z$-isoclinic to   $(S_2, K_2)$ if there exist additive group isomorphisms $\phi :\frac {K_1}{Z(S_1, K_1)}\rightarrow \frac {K_2}{Z(S_2, K_2)}$ such that $\phi \left(\frac {S_1}{Z(S_1, K_1)}\right) = \frac {S_2}{Z(S_2, K_2)}$  and $\psi :[S_1, K_1]\rightarrow [S_2, K_2]$ such that $\psi ([u_1, v_1])=[u_2, v_2]$ whenever $u_i \in S_i$,  $v_i\in K_i$ for  $i = 1, 2$; $\phi (u_1 +  Z(S_1, K_1)) = u_2 +  Z(S_2, K_2)$ and $\phi (v_1 +  Z(S_1, K_1)) = v_2 +  Z(S_2, K_2)$.
Such pair of mappings $(\phi,\psi)$ is called a   $\mathbb Z$-isoclinism between $(S_1, K_1)$ and $(S_2, K_2)$.
\end{definition}
Dutta, Basnet  and Nath  \cite{jutireka} also showed that $\Pr(S_1, R_1) = \Pr(S_2, R_2)$ if the pairs $(S_1,R_1)$ and $(S_2,R_2)$ are  $\mathbb Z$-isoclinic (see Theorem 3.3). In this section, we further generalize this result in the following way.


\begin{theorem}
Let $R_1$ and $R_2$ be two non-commutative rings with subrings $S_1, K_1$ and $S_2, K_2$ respectively. If  $(\phi,\psi)$ is a    $\mathbb Z$-isoclinism between $(S_1, K_1)$ and $(S_2, K_2)$ then
\[
{\Pr}_r(S_1, K_1) = {\Pr}_{\psi (r)}(S_2, K_2).
\]
\end{theorem}

\begin{proof}
By Theorem \ref{com-thm}, we have
\begin{align*}
{\Pr}_r(S_1, K_1)
=&\frac {|Z(S_1, K_1)|}{|S_1||K_1|}\underset{r \in [s_1, K_1]}{\underset{s_1 + Z(S_1, K_1)\in\frac {S_1}{Z(S_1, K_1)}}{\sum}|C_{K_1}(s_1)|}
\end{align*}
noting that $r \in [s_1, K_1]$ if and only if $r \in [s_1 + z, K_1]$ and $C_{K_1}(s_1) = C_{K_1}(s_1 + z)$ for all $z \in Z(S_1, K_1)$.
Now, by  Lemma \ref{lemma1}, we have
\[
{\Pr}_r(S_1, K_1) = \frac {|Z(S_1, K_1)|}{|S_1|}\underset{r \in [s_1, K_1]}{\underset{s_1 + Z(S_1, K_1)\in\frac {S_1}{Z(S_1, K_1)}}\sum} \frac {1}{|[s_1,K_1]|}.
\]
Similarly, it can be seen that
\[
{\Pr}_{\psi (r)}(S_2, K_2) = \frac {|Z(S_2, K_2)|}{|S_2|}\underset{\psi (r)\in [s_2, K_2]}{\underset{s_2 + Z(S_2, K_2)\in\frac {S_2}{Z(S_2, K_2)}}\sum}
\frac {1}{|[s_2,K_2]|}.
\]
Since $(\phi,\psi)$ is a    $\mathbb Z$-isoclinism between $(S_1, K_1)$ and $(S_2, K_2)$ we have  $\frac {|S_1|}{|Z(S_1, K_1)|} = \frac {|S_2|}{|Z(S_2, K_2)|}$, $|[s_1,K_1]| = |[s_2, K_2]|$ and $r \in [s_1, K_1]$ if and only if  $\psi (r) \in [s_2, K_2]$. Hence
\[
{\Pr}_r(S_1, K_1) = {\Pr}_{\psi (r)}(S_2, K_2).
\]
\end{proof}


We conclude this paper by the following result.
\begin{proposition}
Let $R_1$ and $R_2$ be two non-commutative rings with subrings $S_1, K_1$ and $S_2, K_2$ respectively. If $\phi_1 : \frac {S_1}{Z(S_1, R_1)}\to \frac {S_2}{Z(S_2, R_2)}$, $\phi_2 :\frac {K_1}{Z(K_1, R_1)} \to \frac {K_2}{Z(K_2, R_2)}$ and $\psi : [S_1, K_1]\to [S_2, K_2]$ are additive group isomorphisms such that
\[
a_{(S_2, K_2)} \circ (\phi_1 \times \phi_2) = \psi \circ a_{(S_1, K_1)}
\]
where  $a_{(S_i, K_i)}: \frac{S_i}{Z(S_i, R_i)} \times \frac{K_i}{Z(K_i, R_i)} \to [S_i, K_i]$  are well defined maps given by
\[
a_{(S_i, K_i)}(x_i + Z(S_i, R_i), y_i + Z(K_i, R_i)) = [x_i, y_i]
\]
 for all $x_i \in S_i, y_i \in K_i$ and  $i = 1, 2$;  and
\[
(\phi_1 \times \phi_2)(x_1 + Z(S_1, R_1), y_1 + Z(K_1, R_1)) = (x_2 + Z(S_2, R_2), y_2 + Z(K_2, R_2))
\]
whenever
$\phi_1(x_1 + Z(S_1, R_1)) = x_2 + Z(S_2, R_2)$ and $\phi_2(y_1 + Z(K_1, R_1)) = y_2 + Z(K_2, R_2)$. Then
\[
{\Pr}_r(S_1, K_1) = {\Pr}_{\psi (r)}(S_2, K_2).
\]
\end{proposition}
\begin{proof}
For $i=1,2$ let $Z_i := Z(S_i,R_i)$ and ${Z_i}^\prime := Z(K_i,R_i)$. Then we have

\begin{tiny}
\begin{align*}
{\Pr}_r(S_1, K_1)& =\frac {1}{|S_1||K_1|}|\{(x_1,y_1)\in S_1\times K_1 : [s_1, k_1] = r\}|\\
&=\frac {|Z_1||{Z_1}^\prime|}{|S_1||K_1|}\left|\{(x_1 + Z_1, y_1 + {Z_1}^\prime)\in \frac {S_1}{Z_1}\times \frac {K_1}{{Z_1}^\prime}:a_{(S_1,K_1)}(x_1 + Z_1, y_1 + {Z_1}^\prime) = r\}\right|\\
&=\frac {|Z_1||{Z_1}^\prime|}{|S_1||K_1|}\left|\{(x_1 + Z_1, y_1 + {Z_1}^\prime) \in \frac {S_1}{Z_1}\times \frac {K_1}{{Z_1}^\prime} : \psi \circ a_{(S_1, K_1)}(x_1 + Z_1, y_1 + {Z_1}^\prime) = \psi(r)\}\right|\\
&=\frac {|Z_1||{Z_1}^\prime|}{|S_1||K_1|}\left|\{(x_1 + Z_1, y_1 + {Z_1}^\prime) \in \frac {S_1}{Z_1}\times \frac {K_1}{{Z_1}^\prime} : a_{(S_2,K_2)}\circ (\phi_1\times \phi_2)(x_1 + Z_1, y_1 + {Z_1}^\prime) = \psi(r)\}\right|\\
&=\frac {|Z_2||{Z_2}^\prime|}{|S_2||K_2|}\left|\{(x_2 + Z_2, y_2 + {Z_2}^\prime)\in \frac {S_2}{Z_2}\times \frac {K_2}{{Z_2}^\prime}:a_{(S_2, K_2)}(x_2 + Z_2, y_2 + {Z_2}^\prime)=\psi(r)\}\right|\\
&=\frac {|Z_2||{Z_2}^\prime|}{|S_2||K_2|}\left|\{(x_2 + Z_2, y_2 + {Z_2}^\prime)\in \frac {S_2}{Z_2}\times \frac {K_2}{{Z_2}^\prime} : [x_2 + Z_2, y_2 + {Z_2}^\prime] = \psi (r)\}\right|\\
&=\frac {1}{|S_2||K_2|}|\{(x_2,y_2)\in S_2\times K_2 : [x_2,y_2]=\psi(r)\}|\\
&={\Pr}_{\psi(r)}(S_2, K_2).
\end{align*}
\end{tiny}

\end{proof}

\end{document}